\documentclass[a4paper,reqno,10pt]{article}
\usepackage[english]{babel}
\usepackage{amsmath}
\usepackage{amssymb}
\usepackage{url}
\usepackage{tikz}
\usepackage{amsthm}
\usepackage{amsxtra}
\usepackage{mathrsfs}
\usepackage{fancyhdr}
\usepackage{hyperref}

\newtheorem{lemma}{Lemma}[section]
\newtheorem{proposition}[lemma]{Proposition}

\newtheorem{definition}[lemma]{Definition}

\theoremstyle{definition}
\newtheorem{remark}[lemma]{Remark}

\numberwithin{equation}{section}

\newcommand{\wt}{\widetilde}
\newcommand{\om}{\omega}
\newcommand{\erre}{\mathbb{R}}

\newcommand{\pd}[2]{\frac{\partial {#1}}{\partial {#2}}}

\newcommand{\n}{\noindent}
\newcommand{\R}{\mathbb R}
\newcommand{\RE}{\mathbb R}

\newcommand{\M}{\mathcal{M}}

\newcommand{\sech}{\operatorname{sech}}
\newcommand{\arctanh}{\operatorname{arctanh}}
\newcommand{\lf}{\left}
\newcommand{\ri}{\right}
\newcommand{\ve}{\varepsilon}
\newcommand{\al}{\alpha}
\newcommand{\bt}{\beta}

\newcommand{\ga}{\gamma}
\newcommand{\la}{\lambda}
\newcommand{\ome}{\omega}

\newcommand{\ul}{\underline}
\newcommand{\f}{\frac}
\newcommand{\be}{\begin{equation}}
\newcommand{\ee}{\end{equation}}
\newcommand{\beq}{\begin{equation}}
\newcommand{\eeq}{\end{equation}}
\newcommand{\G}{{\mathcal G}}

\newcommand{\tbt}{\tilde\bt}
\newcommand{\tome}{\tilde\ome}

\tikzstyle{nodino}=[circle,draw,fill,inner sep=0pt,minimum size=0.5mm]
\tikzstyle{nodo}=[circle,draw,fill,inner sep=0pt,minimum size=1.5mm]
\tikzstyle{infinito}=[circle,inner sep=0pt,minimum size=0mm]

\title{Stable standing waves
  for a NLS on star graphs as local minimizers of the constrained energy}

\author{Riccardo Adami$^* \! \!$, Claudio
  Cacciapuoti$^\dagger \! \!$, Domenico
  Finco$^\ddagger \! \!$, Diego
  Noja$^\S \!$ 
\\ \ \\{\small   $^*$Dipartimento di Scienze
Matematiche ``G.L. Lagrange'', Politecnico di Torino } \\ {\small
Corso Duca degli Abruzzi, 24, 10129 Torino, Italy} \\ 
{\small e-mail: {{riccardo.adami@polito.it}}}
\\ \ \\{\small   $^\dagger$Dipartimento di Scienza e Alta Tecnologia,
  Universit\`a dell'Insubria} \\
{\small Via Valleggio, 11, 22100 Como, Italy}
\\
{\small e-mail: {{claudio.cacciapuoti@uninsubria.it}}}
\\ \ \\{\small   $^\ddagger$Facolt\`a di Ingegneria, Universit\`a
  Telematica Internazionale Uninettuno }\\
{\small Corso Vittorio Emanuele II, 39, 00186 Roma, Italy}
\\
{\small e-mail: {{d.finco@uninettunouniversity.net}}}
\\ \ \\{\small   $^\S$Dipartimento di Matematica e Applicazioni,
  Universit\`a di Milano Bicocca}\\
{\small Via R. Cozzi, 53, 20125 Milano, Italy}
\\
{\small e-mail: {{diego.noja@unimib.it}}}}

\date{}

\begin{document}

\maketitle

\begin{abstract}

\noindent
On a star graph made of $N \geq 3$ halflines (edges) we consider a
Schr\"odinger equation with a subcritical power-type nonlinearity and
an attractive delta interaction located at the vertex. 
From previous works it is known that there exists a family of standing
waves, symmetric with respect to the exchange of edges, 
that can be parametrized by the mass (or $L^2$-norm) of its elements.
Furthermore, if the mass is small enough, then
the corresponding symmetric standing wave is a ground state
and, consequently, it is orbitally stable. On the other hand, if the mass is above a
threshold value, then the system has no ground state.

\noindent
Here we prove that orbital stability holds for every value of the
mass, even if the corresponding symmetric standing wave
is not a ground state, since it is anyway a {\em local}
minimizer of the energy among functions with the same mass.

\noindent
The proof is based on a new technique that allows to restrict the
analysis to functions made of pieces of soliton, reducing the problem
to a finite-dimensional one. In such a way, we do not need to use 
direct methods of Calculus of Variations, nor linearization procedures.

\end{abstract}





\section{Introduction} \label{sec:introduction}
The subject of nonlinear dynamics on quantum graphs (or {\em
  networks}, see \cite{berkolaiko,kuchment} for an exhaustive introduction)
dates back to the seminal
papers by Ali Mehmeti (see \cite{ali} and references therein), where
the focus was on dispersive properties, and by von Below
\cite{vonbelow}, who first set variational problems on a network. Since then, the interest
increased at a growing rate up to the flourishing of results of the
last decade, motivated by the need for simple models in 
contexts where two features coexist: a basic environment endowed with
branches, junctions, ramifications, and the presence of non-negligible
nonlinear effects. Such models range from quantum optics
(see, e.g., \cite{smi}) to
Bose-Einstein condensation 
(see, e.g., \cite{vidal} and, for a more comprehensive introduction to
physical applications, \cite{noja14}). An important
part of recent results concerns the study of the Nonlinear
Schr\"odinger  Equation (NLS), see e.g. 
\cite{acfn-rmp,acfn-jpa,acfn-aihp,acfn-jde,ast-cv,ast-arxiv,cacciapuoti,
  NPS, quellochefinitosucmp,kevrekidis,matrasulov,hannes}, 
and many of them are concerned with the seek for {\em standing waves},
i.e. solutions to the NLS that preserve the spatial shape and
harmonically oscillate in time (namely, 
solutions of the form $\Phi (t) = e^{i \om t}
\Phi_\omega$, where $\Phi_\omega$ is the space profile, called {\em stationary}
or {\em bound state}), 
or
even for {\em ground states}, i.e., stationary states that minimize the
NLS energy among all functions with the same $L^2$-norm or {\em mass}.

To this regard, in spite of the quasi one-dimensional nature of
networks, the structure of the family of standing waves,
as well as the problem of the existence of a ground state, is far
richer and more complicated than for the NLS on
the line (for recent developments in this direction, see
\cite{ast-cv,ast-arxiv}).

In the present paper we consider a star graph $\G$ made of $N$ ($\geq \, 3$)
halflines that meet one another at the unique vertex $\textsc v$ (see
Fig.\ref{4star}), and
study the dynamics generated on it by the NLS with an attractive
$\delta$-interaction of strength $- \alpha$, $\alpha>0$, placed at $\textsc v$ (for
the precise definitions see formulas \eqref{schrod}, \eqref{accalarenzia}). 
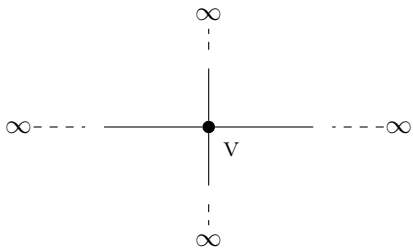
\begin{figure}[h] 
\begin{center}
\begin{tikzpicture}
\node at (-2.5,0) [infinito]  (1) {$\infty$};
\node at (0,0) [nodo] (2) {};
\node at (2.5,0) [infinito]  (3) {$\infty$};
\node at (0,1.5) [infinito]  (4) {$\infty$};
\node at (0,-1.5) [infinito]  (8) {$\infty$};
\node at (1.5,0)  [minimum size=0pt] (5) {};
\node at (-1.5,0) [minimum size=0pt] (6) {}; 
\node at (0,0.9) [minimum size=0pt] (7) {};
\node at (0,-0.9) [minimum size=0pt] (9) {};
\node at (0.3,-0.3) {${{\textsc v}}$};

\draw[dashed] (1) -- (6);
\draw[dashed] (3) -- (5);
\draw[dashed] (7) -- (4);
\draw[dashed] (9) -- (8);
\draw [-] (2) -- (6) ;
\draw [-] (5) -- (2) ;
\draw [-] (2) -- (7) ;
\draw [-] (2) -- (9) ;
\end{tikzpicture}
\end{center}
\caption{A star graph made of four halfline and a vertex.}
\label{4star}
\end{figure}
It is already known
(see \cite{acfn-aihp,acfn-jde}) that for any frequency $\om \in (
{\alpha^2}/ {N^2}, + \infty)$  there exists a unique, real stationary state
$\Psi_\om$ whose associated solution to the NLS oscillates at the
frequency $\om$ and is symmetric under exchange of edges, namely, the 
restriction of $\Psi_\om$ to any halfline gives the same function (see Fig.\ref{fig-psi}). 
\begin{figure}
\begin{center}
\begin{tikzpicture}[xscale= 0.5,yscale=0.8]
\node at (-7.4,0) [infinito]   {$\infty$};
\node at (7.4,0) [infinito]   {$\infty$};
\node at (4.4,2.2) [infinito]   {$\infty$};
\node at (-4,-3.3) [infinito]   {$\infty$};
\draw[-][blue,thick](0,1.2) to [out=340,in=178] (7,0.05);
\draw[-] (0,0) -- (4,0);
\draw[-][blue,thick] (0,1.2) to [out=200,in=2] (-7,0.05);
\draw[-] (-4,0) -- (-0,0);
\draw[-][blue,thick] (0,1.2) to [out=350,in=205] (4,2.1);

\node at (0,0) [nodo] (2) {};

\node at (0,-0.4) {${\textsc v}$};
\draw[dashed] (4,0) -- (7,0);
\draw[dashed] (-7,0) -- (-4,0);
\draw[dashed] (-4,-3) -- (-2,-1.5);
\draw[-] (-2,-1.5) -- (0,0);
\draw[-] (0,0) -- (2,1);
\draw[dashed] (2,1)--(4,2);
\draw[-,blue,thick](-4,-2.98) to [out=40,in=250] (0,1.2);
\draw[dotted] (0,0)--(0,1.2);

\end{tikzpicture}
\end{center}
\caption{A representation of a symmetric stationary state $\Psi_\om$
  on the star graph made of four halflines.}
\label{fig-psi}
\end{figure}
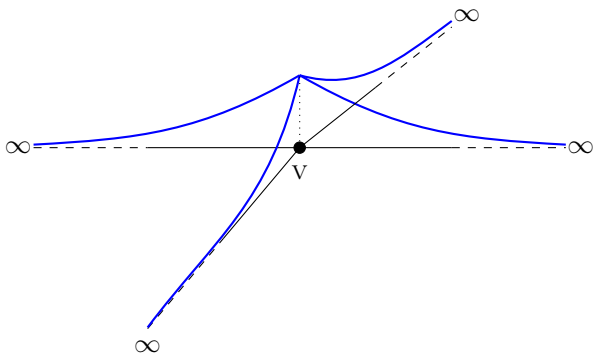
The main result
of this paper is the orbital stability of $\Psi_\om$
and can be expressed as follows:

\medskip

\noindent
{\bf Theorem 1.} {\em
On the star graph $\G$ made of $N \geq 3$ halflines intersecting
one another at the vertex $\textsc v$, consider the 
Schr\"odinger Equation \eqref{schrod} with a focusing nonlinearity of
power $2 \mu + 1$ and a delta interaction at $\textsc v$ with strength
$- \alpha$, with $0 < \mu < 2$ and $\alpha > 0$. 

\noindent
Then, given $\om > \alpha^2 / N^2$, the unique standing wave
$e^{i \om t} \Psi_\om$ 
symmetric
under exchange of edges, is orbitally stable.

}

\medskip
We recall here that orbital stability is Lyapunov stability for orbits
  instead of states. Indeed, in order to hold for the
  dynamics generated by the NLS, that enjoys phase invariance,
Lyapunov stability must be weakened: it cannot hold for states, but it
may hold for orbits. In other words, Theorem 1 establishes that
 a solution remains 
  arbitrarily close {\em to the orbit} of $\Psi_\om$, provided
  that the initial data had been chosen as suitably close {\em to the same orbit}.

With Theorem 1 we
complete the analysis carried out in \cite{acfn-aihp},
where the existence of a ground state at a fixed mass has been
established assuming that the mass  
is smaller than a critical value. In that case, orbital stability follows
from the fact that $\Psi_\om$ is a ground state, so the celebrated general
result by Cazenave and Lions \cite{cl82} applies. On the other
hand, if its mass is larger than the critical threshold, then
$\Psi_\om$ is {\em not} a ground state (see point 4. in Section \ref{section:notation}),
therefore
one is forced to use a criterion for orbital stability for
which it is not necessary to assume that $\Psi_{\om}$ is a ground
state. In fact, Theorem 3 in \cite{gss1} establishes that a stationary
state is orbitally stable if and only if it is a {\em local} minimum
for the energy functional among the functions with the same mass. 
Even though, in general, proving that a function is a local minimum of a functional
is a difficult task, in our case it is possible to exploit the
particular structure of the star graph made of $N$ halflines, and  Th. 4.1 in \cite{ast-cv} which states that  the minimizer of the NLS energy on the halfline under the mass
constraint and a nonhomegeneous Dirichlet condition at the origin is given by a
 unique {\em soliton branch} (we recall its explicit expression in Sec. \ref{section:notation}).  This
fact allows to introduce the so
called {\em multi-soliton transformation}, that maps almost every function
on the graph into a function (\emph{multi-soliton}) made of $N$ pieces
of solitons, one for each 
halfline, in such a way that the mass is preserved and the energy is
lowered. The space of multi-solitons with the same mass, denoted by $\mathcal{M}$ (see Def. \ref{multimani}),  is a {\em finite-dimensional} manifold that contains all the 
stationary states. Thus, proving that a stationary state is a
local minimum in $\mathcal{M}$ 
requires a  finite-dimensional analysis only. Nevertheless, this turns out to
be not an immediate issue, because the study of the sign of the Hessian of the
energy requires a certain degree of explicitness. In order to get it,
one has to further reduce the problem from $N$ halflines to two, and
use the fact that on the real line the orbital 
stability has been already proven by Fukuizumi, Ohta and Ozawa in
\cite{foo08}. Finally, owing to the  continuity of the multi-soliton
transform, one shows that the symmetric stationary state is not only a
local minimum of the energy among the multi-soliton states, but also among 
all states with the same mass.

\smallskip

The statement of Theorem 1 emphasizes the dynamical content of the
result, i.e. the orbital stability. It emerges from the preceding
discussion that 
one can give a variational version of the same result, that highlights
the fact that the examined standing wave
corresponds in 
fact to a {\em local} minimizer of the energy in the appropriate
space, which is in general a highly 
non-trivial goal in the Calculus of Variations. For this reason we
remark that Theorem 1 can be stated also in the following variational
version: 

\smallskip

\noindent
{\bf Theorem 1'.} (Variational version)
{\em
On the star graph $\G$ made of $N \geq 3$ halflines intersecting
one another at the vertex $\textsc v$, consider the 
Schr\"odinger Equation \eqref{schrod} with a focusing nonlinearity of
power $2 \mu + 1$ and a delta interaction at $\textsc v$ with strength
$- \alpha$, where $0 < \mu < 2$ and $\alpha > 0$. 

\noindent
Then, given $\om > \alpha^2 / N^2$, the unique positive bound state
$\Psi_\om$ 
symmetric
under exchange of edges, is a strict (up to phase invariance) local minimizer for the energy
functional associated to the considered evolution equation,
constrained to the manifold of constant $L^2$-norm.
}

\medskip

In \cite{acfn-jde} we treated the dual problem, namely the
minimization of the {\em action} functional on the associated Nehari
manifold. It was proved that in order to have a minimizer, the delta
interaction must be strong enough. The link with the small mass
condition given in \cite{acfn-aihp} can be reconstructed by noting
that, given a frequency $\omega$, the mass of $\Psi_\om$ is a
monotonically decreasing function of the interaction 
strength $\alpha$,
so that the assumption of large $\alpha$ can be thought of as a small
mass hypothesis. Through Theorem 1 one can get rid of every such assumptions:
$\Psi_\om$ is always orbitally stable, irrespective of the mass and of
the interaction strentgh.
Notice that in \cite{acfn-jde} it was proven that the energy
constrained at constant mass has also other stationary points than the
symmetric states above discussed and so the NLS on a star graph admits
several families of non symmetric standing waves. In \cite{acfn-aihp}
it is shown 
that they are excited states, in the sense that their energy is above
the energy of the symmetric state.
They are believed to be saddle points
of the constrained energy, which is consistent with the fact that the
local minimum represented by the symmetric state is above the infimum
of the energy for large mass.

Theorem 1 is already known to hold for $N=2$ too, i.e. for the NLS with an
attractive delta interaction on the line, since in that case
$\Psi_\om$ (which is the unique bound state at its frequency and at
its mass) is always a ground state (see \cite{anv,fukujean,foo08,lacozza}).

The structure of the paper is the following: in Section
\ref{section:notation} we introduce the notation, recall the state of the
art (in particular we explain the result of Th. 4.1 in \cite{ast-cv}), give
some preliminary results, introduce and describe the {\em
  multi-soliton manifold} $\mathcal{M}$. 
In Section \ref{section:transforms} 
we introduce the {\em multi-soliton transformation}, that allows to
reduce the problem to a finite-dimensional setting.  
In
Section \ref{section:results} we prove that $\Psi_\om$ is actually a
local minimum for the restriction of the NLS energy to $\mathcal{M}$. Finally, in Section \ref{section:orbital} we
extend the minimality property of $\Psi_\om$ from the multi-soliton
manifold to the natural space of the functions with finite energy and
constant mass, so accomplishing the proof of Theorem 1.

\medskip

\noindent
{\bf Acknowledgements.} The authors are partially supported by the
FIRB 2012 project ``Dispersive dynamics: Fourier Analysis and
Variational Methods'', code RBFR12MXPO. Furthermore, R.A. is grateful to the PRIN 2012
project ``Aspetti variazionali e perturbativi nei problemi
differenziali nonlineari'' and to the 2015 GNAMPA project ``Propriet\`a
spettrali delle equazioni di Schr\"odinger lineari e nonlineari''. C.C. acknowledges the support of the FIR 2013 project ``Condensed Matter in Mathematical Physics", code RBFR13WAET.

\section{Setting, notation, previous results} \label{section:notation}
We follow in notation the previous papers \cite{acfn-jpa, acfn-epl,acfn-aihp,
  acfn-jde}. The basic environment is
the metric graph $\G$, defined as the star graph made of $N$ halflines
intersecting at their origins, where the unique vertex $\textsc v$ is
located. 
This convention fixes a coordinate
system on $\G$, namely, each edge is identified with the real
nonnegative halfline $[0, + \infty)$ and is endowed with its own
  coordinate (denoted by $x_j$ for the $j$.th edge), while $x_j = 0$
  is the coordinate of $\textsc v$, regardless of $j$. Functions $\Psi
  : \G \longrightarrow \mathbb C$ can be
represented as vectors where the  component on the $j.$th halfline is a scalar function
$\psi_j : [0, +\infty) \longrightarrow \mathbb C$. 
For the sake of clarity, we shall often use the column vector
representation
$$
\Psi \ = \ \left( \begin{array}{c}
\psi_1 \\ \psi_2 \\ \vdots \\ \psi_N
\end{array}
\right).
$$
Besides, the component of $\Psi$ on the $j$.th edge is denoted by
$\psi_j$ or by $(\Psi)_j$.
In general, we use uppercase greek letters for functions defined on
$\G$, and lowercase greek letters for functions acting on the halfline
or on the line.

We denote by $\Psi'$ and $\Psi''$ the vector valued functions with components $\psi_j'$ and $\psi_j''$ respectively, the derivatives taken along the coordinate $x_j$. 

Spaces $L^p(\G)$ are naturally defined as
  the direct sum of $N$ copies of $L^p (0, +\infty)$, while $H^1
  (\G)$
  is the direct sum of $N$ copies of  $H^1 (0, +\infty)$ with the
  additional condition of {\em continuity at the vertex}
\be \nonumber 
\psi_1 (0) = \psi_2 (0) = \cdots = \psi_N (0) = \Psi (\textsc v).
\ee

On $\G$ we consider the dynamics generated by the {\em Schr\"odinger
  equation with power nonlinearity and an attractive point interaction
at the vertex}, formally
\be \nonumber 
i \partial_t \Psi (t) \ = \ - \Delta \Psi (t) - |\Psi (t)|^{2\mu} \Psi
(t) - \alpha \delta_{\textsc
  v} \Psi (t),
\ee
where $\mu> 0$  and the notation $|\Psi|^{2\mu} \Psi$  is understood in vector representation as 
$$
|\Psi|^{2\mu} \Psi \ = \ \left( \begin{array}{c}
|\psi_1|^{2\mu} \psi_1\\ |\psi_2|^{2\mu} \psi_2\\ \vdots \\ |\psi_N|^{2\mu} \psi_N 
\end{array}
\right),
$$
moreover, $\alpha > 0$ and $\delta_{\textsc
  v}$ denotes a delta potential located at the vertex. More
precisely,
\be \label{schrod}
i \partial_t \Psi (t) \ = \ H \Psi (t)  - |\Psi (t)|^{2\mu} \Psi
(t),
\ee
where $H$ is the linear operator defined as 
\be \label{accalarenzia} \begin{split}
D (H) \ = \ & \{ \Psi \in H^1 (\G), \ {\rm s.t.} \ \psi_j \in H^2 (0,
+ \infty), \
\sum_{j=1}^N \psi_j'(0) = -\alpha \Psi (\textsc{v}) \} \\
(H \Psi)_j \ = \ & - \psi_j'' \ .
\end{split}
\ee
Global well-posedness of \eqref{schrod} (for $0 <
\mu < 2$) has been proved first in
\cite{acfn-rmp} for the cubic case, then in \cite{acfn-jde} for the
case of a more
general power, and for the case $N=2$ (real line) with a general point
interaction in \cite{an-jpa}. Furthermore, it has been also proved that the
$L^2$-norm, or {\em mass}
\be \nonumber 
Q (\Psi) \ = \ \| \Psi \|_{L^2 (\G)}^2
\ee
 and the energy
\begin{equation} \label{energy}
E (\Psi, \G) \ = \ \f 1 2 \| \Psi' \|_{L^2(\G)}^2 - \f 1 {2 \mu + 2} \| \Psi
\|_{L^{2 \mu + 2}(\G)}^{2 \mu + 2} - \f \alpha 2 | \Psi (0) |^2 
\end{equation}
are conserved by the flow.

We shall occasionally make use of the functionals $E (\cdot, \R^+)$
and $E (\cdot, \R)$, that share the formal expression of
\eqref{energy} but are evaluated on functions on the halfline and on
the line, respectively.

As explained in Sec.\ref{sec:introduction}, we aim at proving that
$\Psi_\om$ is a local minimizer for $E( \cdot, \G)$ with the mass
constraint, so, chosen $M > 0$, our reference space is
\be \nonumber 
H^1_M (\G) \ : = \ H^1(\G) \cap \{ Q (\Psi) = M \}. 
\ee

Let us now recall some preliminary notions, 
together with some further basic definitions and notation.

\smallskip

1. For every $\omega > 0$, the {\em soliton}
\be \nonumber 
\phi_\om (x) = [ (\mu + 1) \omega]^{\f 1 {2 \mu}} 
\sech^{\f 1 \mu} (\mu \sqrt \om x)
\ee
is the unique positive square-integrable 
solution to the stationary NLS equation
$$
\varphi'' (x) + \varphi^{2 \mu + 1} (x) \ = \ \omega \varphi (x).
$$
As a consequence,
the function $e^{i \om t} \phi_\omega$ is a standing wave for the NLS on the
line with power nonlinearity $2 \mu + 1$. Besides, for any $\omega>0$ the \emph{mass} of the soliton  is given by  
\begin{equation}\label{solmass}
\| \phi_\om \|_{L^2 (\R)}^2 = 2\f{(\mu+1)^{\f 1 \mu }  }{\mu} \ome^{ \f 1 \mu - \f
  1 2} \int_0^1 (1-t^2)^{\f 1 \mu -1} dt ,
\end{equation}
and is a monotonically increasing function of $\om$.

\smallskip

2. As proved in \cite{acfn-jde}, for every $\alpha > 0$, the 
unique positive stationary solution to
Eq. \eqref{schrod} with frequency $\omega > \alpha^2 / N^2$, symmetric under exchange
of edges, is given by
\beq
\lf(\Psi_{\omega}\ri)_i(x_i) = 
\phi_\ome(x_i + \zeta), \quad  i=1, \ldots, N
\nonumber 
\eeq 
with 
\beq 
\zeta = \f{1}{\mu \sqrt{\ome}} \arctanh
\lf(\f{\al}{N \sqrt{\ome}} \ri) .
\label{states2} 
\eeq 
Notice that $\zeta > 0$, so that 
every halfline hosts a {\em soliton tail}, which is a monotonically
decreasing function (see Fig.\ref{fig-psi}).

\smallskip
3. The {\em mass function}
\begin{equation}\label{darker}
M (\omega) \ : = \ Q
(\Psi_{\omega})  \ = \ N \f {(\mu + 1)^{\f 1 \mu}} \mu \om^{\f 1 \mu -
  \f 1 2}
\int_{\f \alpha {N \sqrt \om}}^1 (1 - t^2)^{\f 1 \mu - 1} dt,
\end{equation}	
(see formula (5.1) in \cite{acfn-aihp})
 is strictly monotonically increasing, and
ranges from $0$ (excluded)
to
$+\infty$ as $\omega$ goes from $\alpha^2/N^2$ (excluded) to $+\infty$. 
Then, in the same way as for the soliton on the real line,  $\omega$ can be interpreted as a relabelling of
the mass, and for every positive $M$ there exists exactly one
symmetric stationary state $\Psi_{\omega}$ 
with mass $M$ (see \cite{acfn-aihp}).
 
\smallskip
4. As mentioned in Sec.\ref{sec:introduction}, there are some
values of $M$ such that the corresponding stationary state $\Psi_\om$ is
{\em not} a ground state. Indeed, in \cite{acfn-jde}  we exhibited a
sequence $\Phi_n \in H^1_M (\G)$ s.t. $E(\Phi_n, \G) \to E (\phi_{\om_\R}, \R)$ as
$n$ goes to infinity (see formula (3.12) in \cite{acfn-jde}), where
$\om_\R$ is the unique value of $\om$ such that $\| \phi_\om \|_{L^2
  (\R)}^2 = M$, see Eq. \eqref{solmass}. Roughly
speaking, such a sequence is supported on a single edge and
asymptotically reconstructs a soliton at infinity. Therefore, in order
for $\Psi_\om$ to be a
ground state, it must be $E(\Psi_\om, \G) \leq E (\phi_{\om_\R},
\R)$. By explicitly computing the involved energies, see also formulas (4.10) and (4.12) in \cite{acfn-aihp},
such inequality can be rewritten as
\be \label{violated}
(2 - \mu ) \om_\R M \ \leq \ (2 - \mu) \om M + \alpha \mu (\mu + 1)^{\f 1
  \mu} \left( \om - \f {\alpha^2}{N^2} \right).
\ee
Furthermore, since $\|\Psi_\omega\|_{L^2(\G)}^2 = M = \|\phi_{\omega_\RE}\|^2_{L^2(\RE)}$, and  from Eqs. \eqref{solmass} and \eqref{darker}, one has the identities  
\be \label{massse} \begin{split}
M  \ & \ =  2\f{(\mu+1)^{\f 1 \mu }  }{\mu} \ome_\RE^{ \f 1 \mu - \f
  1 2} \int_0^1 (1-t^2)^{\f 1 \mu -1} dt  \\
& \  = \ N \f{(\mu+1)^{\f 1 \mu
}  }{\mu}  \ome^{ \f 1 \mu - \f 1 2}  \int_{ \frac{\al}{ N
    \sqrt{\ome}}}^1 (1-t^2)^{\f 1 \mu -1} dt. 
\end{split} 
\ee
Hence, for $M$ large, the l.h.s. of \eqref{violated} is of order $M^{\f {\mu + 2} {2 - \mu}}$, as
the first term in the r.h.s., while the second term in the r.h.s. is
of order $M^{\f {2} {2 - \mu}}$, then it can be
neglected. Expliciting $\om_\R$ and $\om$ as functions of $M$ in
\eqref{massse}, one has that inequality \eqref{violated} amounts to 
$$
 \int_0^1 (1-t^2)^{\f 1 \mu -1} dt \ \geq \ \f N 2  \int_{ \frac{\al}{ N
    \sqrt{\ome}}}^1 (1-t^2)^{\f 1 \mu -1} dt,
$$
that is violated for $M$ large, as $\om$ becomes large too. Then,
large mass implies that $\Psi_\om$ is not a ground state, and, since
by Lemma 5.2 in \cite{acfn-aihp}, $\Psi_\om$ is the minimizer of the
energy among all stationary states, one concludes that there is no
ground state if the mass exceeds a critical
threshold.

\medskip

The last part of the present section is devoted to the introduction of
the finite-dimensional manifold to which we shall reduce the problem.

We preliminarily recall
Theorem 4.1 of \cite{ast-cv}, which establishes that, given $a, m >
0$, there exists a {\em unique} couple $\omega > 0, \xi \in \R$ such that
\begin{equation}\label{giveup} \int_0^{+\infty}\phi^2_\omega (x +
\xi) \, dx  = m, \qquad \phi_\omega (\xi) = a.\end{equation}
Furthermore, the function $\phi_\omega ( \cdot + \xi)$ minimizes the
energy 
\[\f12\|\phi'\|_{L^2(\RE^+)}^2-\f1{2\mu+2}\|\phi\|_{L^{2\mu+2}(\RE^+)}^{2\mu+2}\] among the functions in $H^1 (\R^+)$ with mass $m$ and whose value at zero equals 
$a$ (Dirichlet constraint), hence, within such class of functions, it is also the minimizer of $E(\cdot,\RE^+)$.  

\noindent
By this result one can introduce two functions 
$\omega = \omega (m,a), \xi = \xi (m,a)$, that give the value of the
soliton parameters $\om$ and $\xi$ such that $\phi_\om (\cdot + \xi)$
minimizes the functional $E (\cdot, \R^+)$ with both mass and
Dirichlet constraints given by Eq. \eqref{giveup}.

As we will show in Sec.\ref{section:transforms}, the existence of such
functions is the starting point for the reduction of the problem to a
finite-dimensional manifold, whose definition is:

\begin{definition} \label{multimani}
Fixed $M > 0$,
we call {\em multi-soliton manifold} of mass $M$,  and denote  by $\M$,
the subset of $H^1_M (\G)$ made of functions whose restriction at every halfline of
$\G$ gives a piece of soliton.
\end{definition}

\begin{remark}
Notice that all functions in $\M$ are positive. 
\end{remark}

\begin{remark} \label{2.2}
Every element of $\M$ has the following form:
\begin{align} \label{drei}
\Phi_{\ul m,a} = 
& \left( \begin{array} {c}
  \phi_{\omega(m_1, a)}
( \cdot + \xi (m_1, a)) \\  \phi_{\omega(m_2, a)}
( \cdot + \xi (m_2, a))\\  \vdots \\
\phi_{\omega (M -\sum_{j=1}^{N-1} m_j, a)}
( \cdot + \xi 
  (M -\sum_{j=1}^{N-1} m_j, a)) 
\end{array}
\right)
\end{align}
where

\begin{itemize}
 
\item  we used the notation $\ul m : = (m_1, m_2, \dots m_{N-1})$,
  with $m_i > 0$ and $\sum_{i=1}^{N-1} m_i < M$;

\item by the definition of the functions $\om$ and $\xi$,
$m_i$ is the mass located on the $i$.th edge, i.e.,
\begin{equation}\label{spell1}
\int_0^\infty \phi_{\omega(m_i,a)}^2(x+\xi(m_i,a)) \,  dx = m_i,
\end{equation} 
and
 $a$
is the value attained by $\Phi_{\ul m, a}$ at $\textsc v$, in particular,  for
all $i = (1, \dots, N-1)$ 
\begin{equation}\label{spell2}
\phi_{\omega(m_i,a)}(\xi(m_i,a))  = a.
\end{equation}
\end{itemize}
Notice that the global mass constraint $M (\Phi_{\ul m, a}) = M$ is guaranteed by the last compoment of
the vector in the r.h.s. of \eqref{drei}.

Since the function $\Phi_{\ul m, a}$ depends on
$N$ parameters $m_1, \dots, m_{N-1},a$, $\M$ is a $N$-dimensional
submanifold of $H^1_M (\G)$.
\end{remark}

\begin{remark}
The real symmetric stationary state $\Psi_{\om}$ with $Q(\Psi_\om) = M$, belongs to
$\M$. Indeed, it is immediately seen that
$
\Psi_{\om} = \Phi_{\wt {\ul m}, \widetilde a},
$
where the vector $\wt {\ul m}$ reads
$\wt m_1 = \dots = \wt m_{N-1} = \f M N$
and $\widetilde a$ is  defined as the unique solution to the transcendental equation 
$$
 \widetilde a =  \phi_{\om(M/N, \widetilde a)}(\xi(M/N, \widetilde a))
            \ = \ [(\mu+1)\omega (M/N, \widetilde a)]^{\frac1{2\mu}} \left( 1 - \f {\alpha^2}
{N^2 \omega(M/N, \widetilde a)}  \right)^{\f 1 {2 \mu}},
$$
where we observed that $\xi(M/N, \widetilde a) = \zeta$ and then used \eqref{states2}.
\end{remark}
   
\medskip

When dealing with functions in the manifold $\M$, the energy
functional becomes a real-valued function of $N$ real, positive
variables. We emphasize this change of point of view by introducing
the {\em reduced energy} $E_r$ as 
\be \label{eenne} \begin{split}
E_r \ :  \ \{ (\ul m, a) & \in (0, + \infty)^{N}, \  {\rm{s.t.}}
\  \sum_{i=1}^{N-1} m_i <  M \} \to \R \\ 
E_r ( \underline{m}, a) \ : = & E (\Phi_{\ul m,a}, \G).
\end{split}
\ee

\noindent
Let us define the
function $F : [0,M] \times \R^+ \to \RE$ as
\be \label{effe} \begin{split}
F (m, a) \ = \ &  \frac12 \|\phi_{\ome(m, a)}'(\cdot+\xi(m,
a))\|_{L^2(\RE^+)}^2 \\ & -
\frac1{2\mu+2} \|\phi_{\ome(m, a)}(\cdot+\xi(m, a))\|_{L^{2\mu+2}(\RE^+)}^{2\mu+2}
-\frac{\al}{2N} a^2 
\end{split}
\ee
so that the function $E_r$
decomposes as follows
\be \label{energydec}
E_r ( \underline{m}, a) \ = \ 
\sum_{i=1}^{N-1} F (m_i, a) +  F (M - \sum_{j=1}^{N-1} m_j, a).
\ee

\section{Multi-soliton transformation}\label{section:transforms}
The first step of the proof of Theorem 1 consists in reducing the problem to a
finite-dimensional manifold, made of functions obtained by gluing
together pieces of soliton. To this aim, we start by transforming
every function in $H^1_M (\G)$ that does not vanish at $\textsc v$ into
a function of   $\mathcal M$, 
in such a way that the
value at $\textsc v$ and the mass at any edge are preserved.

Owing to Theorem 4.1 in \cite{ast-cv} and to the definition of the
functions $\om$ and $\xi$, introduced in Sec.\ref{section:notation}, it
is possible to give the following 

\begin{definition} \label{soltrans}
The {\em soliton transformation} $S \eta \in H^1 (\R^+)$ of a function $\eta \in
H^1 (\R^+)$ such that $\eta (0) \neq 0$, is defined as 
$$ S \eta \ : = \ \phi_{\omega (m, a)} ( \cdot + \xi
(m, a)), $$ 
where we denoted $m = \int_{\erre^+} | \eta (x) |^2 \, dx$, $a = |
\eta (0)|$.
\end{definition}
By Theorem 4.1 in \cite{ast-cv}, the function $S \eta$ is unique and 
\[\f12\|(S\eta)'\|_{L^2(\RE^+)}^2-\f1{2\mu+2}\|S\eta\|_{L^{2\mu+2}(\RE^+)}^{2\mu+2}\leq\f12\|\eta'\|_{L^2(\RE^+)}^2
- \f1{2\mu+2}\|\eta\|_{L^{2\mu+2}(\RE^+)}^{2\mu+2},\]            
hence, since $(S\eta)(0) = |\eta(0)|$,   Theorem 4.1 in \cite{ast-cv}
implies 
\begin{equation} \nonumber 
E (S \eta, \R^+) \ \leq \ E (\eta, \R^+),
\end{equation}
where equality holds if and only if $\eta = e^{i \theta} \phi_\om (
\cdot + \xi )$, for some values of $\theta, \om$, and $\xi$.

Furthermore, notice that the soliton transformation acts trivially on
pieces of soliton, namely $S \phi_\om (\cdot + \xi) =  \phi_\om (\cdot + \xi)$.

Finally, the soliton transformation is defined
for $a > 0$ only. However, we do not need to cover the case $a = 0$.
The only property we shall use is the continuity of $S$ at the
stationary state $\Psi_{\om}$.
 
\begin{proposition}
The soliton transformation $S$ is continuous from the space $H^1_M (\G)$
to itself. 
\end{proposition}

\begin{proof}
We decompose $S$ in three steps and show continuity at every step.
Schematically,
\begin{equation} \nonumber 
\eta \stackrel{S_1} \mapsto (a,m)  \stackrel{S_2}\mapsto
(\omega, \xi)  \stackrel{S_3} \mapsto \phi_\omega (\cdot + \xi),
\end{equation}
where $a = | \eta (0)|$, $m = \int_0^{+\infty} |\eta (x) |^2 dx$.

First, $S_1$ is continuous because the pointwise value and the
$L^2$-norm are continuous in $H^1 (\R^+)$.

\noindent
As already stated, the fact that the map $S_2$ is well-defined follows
from Theorem 4.1 in \cite{ast-cv}. The proof that it is also
continuous can be made
by closely following the proof of
Theorem 4.1 in \cite{ast-cv}. We sketch the procedure since, due to
differences in the  notation, this step could not be straightforward.

By  using the scaling
 property 
$$
\phi_\omega(x) = \omega^{\frac{1}{2\mu}} \phi_1(\sqrt \omega x )
$$
 in the identities  \eqref{spell1} and \eqref{spell2} we obtain 
\begin{equation}\label{keep0}
 \int_0^\infty \phi_1^2(x+\sqrt \omega \xi) dx = m \omega^{-\frac{2-\mu}{2\mu}}  \qquad \text{and} \qquad   \phi_1(\sqrt \omega \xi) = a \omega^{-\frac{1}{2\mu}}. 
\end{equation}
 Putting them together one obtains the identity  \begin{equation}\label{keep1}g(\sqrt\omega\xi)=\frac{m}{a^{2-\mu}},\end{equation}  with
\begin{equation} \nonumber 
g(z)=\big(\phi_1(z)\big)^{-(2-\mu)}\int_0^\infty\phi_1^2(x+z) dx.\end{equation}  
The function $g$ is continuous and strictly monotonically
decreasing (for the proof of this statement we refer to
the proof of Th. 4.1 in \cite{ast-cv} again). We remark that the definition
of the function $\phi_1$ in \cite{ast-cv} is  slightly different from
ours,  basically the two definitions involve different scalings of the
hyperbolic secant. This is due to the fact that in \cite{ast-cv} the
authors parametrize solitons through the mass instead of the
frequency. In spite of that, the argument in Th. 4.1 still  applies,
being based only the asymptotic, monotonicity,  and log-concavity
properties of $\sech$. 
 
 By \eqref{keep1}, we infer that the quantity $\sqrt\omega \xi$ is a continuous
function of $m$ and $a$.  Moreover, for fixed $m$ and $a$ there exists
a unique value of $\sqrt\omega \xi$ such that  identity  \eqref{keep1}
is satisfied. Furthermore, fixed the quantity $\sqrt{\omega} \xi $,
there exists a unique $\omega$ such that the second identity in
\eqref{keep0}  is satisfied, and such a $\omega$ is a
continuous function of $a$ and $m$. As a consequence $\xi$ is a continuous function of $a$ and $m$ too, and
this proves the continuity of $S_2$. 

In order to prove
the continuity of $S_3$, let us fix the couple $(\om, \xi)$ and prove
that
\be \label{ziel}
\lim_{\om_1 \to \om, \xi_1 \to \xi} \| \phi_{\omega} (\cdot + \xi)
- \phi_{\omega_1} (\cdot + \xi_1) \|_{H^1 (\R^+)} \ = \ 0.
\ee
Preliminarily, we use the triangular inequality
\be \label{triangle} \begin{split} &
\| \phi_{\omega} (\cdot + \xi)
- \phi_{\omega_1} (\cdot + \xi_1) \|_{H^1 (\R^+)}^2 \\
\leq & \ 2 \| \phi_{\omega} (\cdot + \xi)
- \phi_{\omega} (\cdot + \xi_1) \|_{H^1 (\R^+)}^2
+ 2 \| \phi_{\omega} (\cdot + \xi_1)
- \phi_{\omega_1} (\cdot + \xi_1) \|_{H^1 (\R^+)}^2  
\\
\leq & \ 2 \| \phi_{\omega} (\cdot + \xi)
- \phi_{\omega} (\cdot + \xi_1) \|_{H^1 (\R^+)}^2 
+ 2 \int_0^{+\infty} |\phi_{\om} (x + \xi_1) - \phi_{\om_1} (x +
\xi_1)|^2 dx \\ & \ + 2 \int_0^{+\infty} |\phi_{\om}' (x + \xi_1 ) -
\phi_{\om_1}' (x + \xi_1)|^2 
dx \\
\leq & \ 2 \| \phi_{\omega} (\cdot + \xi_1)
- \phi_{\omega} (\cdot + \xi) \|_{H^1 (\R^+)}^2 +
2 \int_{-\infty}^{+\infty} |\phi_{\om} (x) - \phi_{\om_1} (x)|^2 dx 
\\ & \ + 2 \int_{-\infty}^{+\infty} |\phi_{\om}' (x) - \phi_{\om_1}' (x)|^2
dx
\end{split}
\ee
Concerning the first term in the r.h.s., a straightforward computation gives
$$ 
\| \phi_{\omega} (\cdot + \xi) - \phi_{\omega} (\cdot + \xi_1) \|_{L^2(\R^+)}^2
\ \leq \ | \xi - \xi_1 |^2 \| \phi_{\om}'\|_{L^2(\R^+)}^2,
$$ 
and, repeating the same computation for the derivatives, one gets
\be 
\| \phi_{\omega} (\cdot + \xi) - \phi_{\omega} (\cdot + \xi_1 )
\|_{H^1(\R^+)}^2 \  \leq \ | \xi - \xi_1 |^2 \| \phi_\om \|_{H^2
  (\R^+)}^2 . 
\ee
So
the first summand in the r.h.s vanishes as $\xi_1$ approaches
$\xi$. For the second term in the r.h.s. of \eqref{triangle},
first observe that, assuming $\om / 2 \leq
\om_1 \leq 2 \om$,
\begin{equation} \nonumber 
\begin{split}
&
|\phi_{\om} (x ) - \phi_{\om_1} (x)|^2 \ \leq \ 2
|\phi_{\om} (x )|^2 + 2 |\phi_{\om_1} (x)|^2 \\
\ \leq \ &  C \left( \omega^{\f 1 \mu} e^{-2 \sqrt{\om} | x 
  |} + {\omega_1}^{\f 1 \mu} e^{-2 \sqrt{\om_1} | x |} \right) \ \leq \ C
\om^{\f 1 \mu}  e^{ - \sqrt{\om} | x |},
\end{split}
\end{equation}
and the last quantity is an integrable function in the variable $x$. Analogously,
for the last term in the r.h.s. of inequality \eqref{triangle}, one gets
\begin{equation} \nonumber 
\begin{split}
&
|\phi_{\om}' (x) - \phi_{\om_1}' (x)|^2 \ \leq \ 2
|\phi_{\om}' (x )|^2 + 2 |\phi_{\om_1} '(x)|^2 \\
\ \leq \ &  C \left( \omega^{1+\f 1 \mu} e^{-2 \sqrt{\om} | x 
  |} + \omega_1^{1+ \f 1 \mu} e^{-2 \sqrt{\om_1} | x |} \right) \ \leq \ C
\om^{1+\f 1 \mu}  e^{ - \sqrt{\om} | x |}.
\end{split}
\end{equation}
Then, by dominated convergence theorem, one has that the last two
terms in \eqref{triangle} vanish as $\om_1$ goes to $\omega$, thus
\eqref{ziel} is proved and the proof is complete.

\end{proof}
We can now introduce the multi-soliton trasformation $\Sigma$ as the natural
generalization of the soliton transformation $S$ to the star graph $\G$.

\begin{definition} \label{multisol}
Given a function  $\Phi \in H^1_M (\G)$ such that $\Phi (\textsc v) \neq 0$,
the {\em multi-soliton} transformation $\Sigma \Phi$ of $\Phi$ is the
function defined on $\mathcal M$ as
\be \nonumber
(\Sigma \Phi)_j \ : = \ S \Phi_j,
\ee
where $S$ is the soliton transformation 
introduced in Definition \ref{soltrans}.
\end{definition}

\begin{remark} \label{propsigma}
The multi-soliton transformation $\Sigma$ inherits the following
properties from the soliton transformation $S$: 
\begin{enumerate}
\item $\Sigma$ is continuous from the space of the functions in $H^1_M (\G)$
that do not vanish at the vertex, to $H^1_M (\G)$.
\item $\Sigma$ preserves the mass distribution on the edges of the
  star graph, i.e. $\|(\Sigma\Phi)_j\|_{L^2(\RE^+)}^2 =
  \|(\Phi)_j\|_{L^2(\RE^+)}^2$, and the absolute value of the function in the
  vertex, namely $\Sigma\Phi(\textsc{v}) = |\Phi(\textsc{v})|$.   
\item For every $\Phi$ in the domain of $\Sigma$, 
\be \label{energy-decrease}
E (\Sigma \Phi, \G)
  \leq E (\Phi, \G),
\ee 
where equality holds if and only if $\Phi \in \mathcal M$, up to a
constant phase factor. 
\item The multi-soliton transformation $\Sigma$ acts trivially on
  $\mathcal M$.  In particular, $\Sigma
  \Psi_{\om} = \Psi_{\om}$.
\end{enumerate}
\end{remark}

\section{Local minimality of $\Psi_\om$ in $\M$} \label{section:results}
Here we treat the finite-dimensional problem of the local minimality
of $\Psi_\om$ in the manifold $\M$. In this section we always refer to
the reduced energy $E_r$ (see Definition \ref{eenne}) and then, when possible, avoid any reference to
functions, that can be replaced by points of $\R^N$. In particular,
according to Remark \ref{2.2}, the bound state $\Psi_\om$ corresponds
to the point $\widetilde P : = (\wt{\ul m}, \wt a) = (M/N, \dots, M/N,
\widetilde a)$.
However, in order to prove even this finite-dimensional
minimality, along the proof we must come back to function
representation and make use of a well-known result of the
Grillakis-Shatah-Strauss theory on stability of standing waves (see
Theorem 3.4 in \cite{gss1}), that we straightforwardly apply in order
to get inequality \eqref{pre-claim}.

\begin{proposition} \label{finitlocal}
Fixed $M >0$, the point $\wt P$ is a strict local minimum for the function $E_r
(\ul m, a)$ defined in \eqref{eenne}.
\end{proposition}

\begin{proof}
First, notice that $\wt P$ is an internal point for the domain of
$E_r$, therefore it suffices to show that $\wt P$ is a stationary point
for $E_r$ and that the Hessian matrix of $E_r$ evaluated at $\wt P$ is
positive definite.

The fact that $\Psi_{\om}$ is a stationary point immediately gives
that $\widetilde P$ is a stationary point for $E_r$, so we turn to the
study of the sign of the Hessian matrix.

\noindent
By straightforward calculations, 
\begin{eqnarray} \nonumber
\f{\partial^2 E_r} {\partial m_i\partial m_j} (\widetilde P)& = & (1 +
\delta_{ij})
\f{\partial^2
  F} {\partial m^2} \left( \f M N, \wt a \right) \\
\label{straight}
\f{\partial^2 E_r}{\partial a^2}(\widetilde P)& = &  N \f{\partial^2 F
    }{\partial a^2} \left( \f M N, \wt a \right) \\ \nonumber
\f{\partial^2 E_r}{\partial a\partial m_i}(\widetilde P)& = & 0,
\end{eqnarray}
with $1\leq i,j \leq N-1$.

\noindent
Consequently, the Hessian matrix of $  E_r$ computed at
$\widetilde P$ is a $N \times N$ block matrix. The high $(N-1) \times
(N-1)$ left block, that
we call $\mathbb H_1$, 
reads
$$
\mathbb H_1 : = \f{\partial^2
  F} {\partial m^2} \left( \f M N, \wt a \right) ( \mathbb J +
\mathbb I)
$$
where $\mathbb J$ is the matrix with all elements equal to one. By
elementary linear algebra, one immediatly finds that $\mathbb H_1$ 
has two eigenvalues: $N \f{\partial^2
  F} {\partial m^2} \left( \f M N, \wt a \right)$ with multiplicity
$1$, and $\f{\partial^2
  F} {\partial m^2} \left( \f M N, \wt a \right)$ with multiplicity
$N-2$. Therefore, $\mathbb H_1$ is positive definite if and only if 
\be \label{eins}
\f{\partial^2
  F} {\partial m^2} \left( \f M N, \wt a \right) > 0.
\ee
The second diagonal block, denoted by $\mathbb H_2$, is $1 \times 1$ and reduces
to $ N \f{\partial^2 F
    }{\partial a^2} \left( \f M N, \wt a \right)$, so it is positive defined if and only if 
\be \label{zwei}
 \f{\partial^2 F
    }{\partial a^2} \left( \f M N, \wt a \right) > 0.
\ee
Thus, in order to  prove that $\widetilde P$ is actually a local minimum of the
function $E_r$, one has to
prove \eqref{eins} and \eqref{zwei}.

To prove \eqref{eins}, 
consider the curve in the manifold $\M$ given by $\Phi_{\ul{m}(t),\wt a}$, with
$$ m_1 (t) = \f M N + t, \ m_2(t) = \dots = m_{N-1} (t) = \f M N,
 \qquad t \in (-\varepsilon, \varepsilon)$$
and define the function 
$$
f (t) : = E_r (\ul{m}(t), \wt a).
$$
By
definition of partial derivative and using \eqref{straight}, one has
$$
f'' (0) \ = \ \f{\partial^2 E_r} {\partial m_1^2} (\widetilde P
  ) \ = \ 2 \f{\partial^2
  F} {\partial m^2} \left( \f M N, \wt a \right).
$$
So condition \eqref{eins}  reduces to $f'' (0) > 0$. 
Now, let us write the curve in
terms of elements of $\M$:
\begin{align} \nonumber
\Phi_{\ul m(t),\wt a} = 
& \left( \begin{array} {c}
  \phi_{\omega(M/N+t,\wt a)}
( \cdot + \xi (M/N + t, \wt a)) \\  \phi_{\omega(M/N,\wt a)}
( \cdot + \xi (M/N ,  \wt a))\\  \vdots \\ \phi_{\omega(M/N-t,\wt a)}
( \cdot + \xi 
  (M/N - t,  \wt a)) 
\end{array}
\right)
\end{align}
It transpires that $\Phi_{\ul m(t),\wt a}$
varies with $t$ in
the first and the $N$.th components only. Since the total mass is conserved,
as $t$ changes {\em there is a transfer of
mass from one edge to the other without changing the value at the
vertex}. We define the operator 
$$\tau :  L^2 ({\mathcal{G}}) \longrightarrow L^2 (\R)$$
acting as $\tau \Xi \ = \ \eta$, with
\be \nonumber 
\eta (x) \ = \ \chi_{\R^+} (x) (\Xi)_{1} (x) + \chi_{\R^-} (x)
(\Xi)_{N} (-x),
\ee
where, as usual, we denoted by $(\Xi)_j$ the component of the wave function
$\Xi$ on the $j$.th edge.
In other words, $\eta$ is the function on the line obtained from $\Xi$
by matching together  edges $1$ and $N$ and neglecting all the others. Therefore, 
\begin{eqnarray} & & \nonumber
f(t) - E_r (\widetilde P) \ = \ E_r (\ul{m}(t),
  a) - E_r (\widetilde P) \\ \nonumber & = & \sum_{i = 1}^{N-1}
F ( m_i (t), \wt a) + F \left( \f M N - t, \wt a \right)
- N  F \left( \f M N, \wt a \right) \\ \nonumber
& = &  F \left( \f M N + t, \wt a \right) +  F \left( \f
M N - t, \wt a \right)  - 2 F   \left( \f M N, \wt a
\right) \\ \nonumber
& = & E_{2 \alpha / N} ( \tau \Phi_{\ul m (t), \wt a}, \R ) 
- E_{2 \alpha / N} ( \tau \Psi_\om,\R ), \\ \label{gsss}
\end{eqnarray}
where we exploited the definition \eqref{effe} of the function $F$ and
introduced the functional
\be \nonumber 
\begin{split}
E_{2 \alpha / N} (u, \R) \ : =  & \ \f 1 2 \| u' \|^2_{L^2 (\R)} - \f 1 {2 \mu +
  2}
 \| u \|^{2\mu+2}_{L^{2 \mu + 2} (\R)}
  - \frac{\alpha} N | u (0) |^2
\end{split}
\ee
acting on $H^1 (\R)$. Notice that $E_{2\alpha / N} (\cdot, \R)$ is the
functional representing the energy associated to the focusing Schr\"odinger
equation  with power nonlinearity $2 \mu + 1$ and a delta
interaction placed at the origin with strength $2 \alpha / N$ (see \cite{foo08} and \cite{an-jpa}).

Furthermore,
$$
(\tau \Psi_{\omega}) (x)
\ = \ [ (\mu + 1) \om ]^{\f 1 {2 \mu}} \cosh^{-\f 1 \mu} \left(
\mu \sqrt \om (|x| + \zeta) \right),$$
%
%
with $\zeta$ given by Eq. \eqref{states2},   which is {the ground state of the NLS} on the line with a delta
interaction located at the origin, of strength
$\alpha' = \f {2 \alpha} N$. In that case
(see \cite{foo08} and \cite{anv}), such a state is known to be a stable global minimizer of
the constrained problem, so it falls into the scope of Theorem
3.4 of \cite{gss1}. In the notation of that theorem, $T(s(u)) = 1$
since we are dealing with real functions. Therefore, for
$\varepsilon$ sufficiently small, the theorem yields
\begin{eqnarray} \nonumber
f (t) - E_r (\widetilde P) & = &
 E_{2 \alpha / N} ( \tau \Phi_{\ul m (t), \wt a} , \R) 
- E_{2 \alpha / N} ( \tau \Psi_{\om} , \R) \\
& \geq & c \| \tau \Phi_{\ul{m}(t),
  \wt a} - \tau \Psi_{\om} \|_{L^2(\R)}^2, \label{pre-claim}
\end{eqnarray}
for any $t \in (-\varepsilon, \varepsilon)$. We claim that  there exists a positive constant $c$ such that, for $\ve$ small enough, 
\begin{equation}\label{claim}
\| \tau \Phi_{\ul{m}(t),
  \wt a} - \tau \Psi_\om \|_{L^2(\R)}^2 \geq c\, t^2.
\end{equation}
 This gives the bound $f (t) - E_r (\widetilde P) \geq c
 \,t^2$, which, together with $f (0) - E_r (\widetilde P)
 =0$ and $f' (0) = 0$, implies $ 
f'' (0) \geq  c > 0$, and concludes the proof of  \eqref{eins}.

\noindent
It remains to prove the claim \eqref{claim}. To this aim, we write 
\[ \begin{aligned} 
&\| \tau \Phi_{\ul{m}(t),  \wt a} - \tau \Psi_\om \|_{L^2(\R)}^2 \\ 
= &  \int_0^\infty \big(\phi_{\omega(M/N+t, \wt a)} ( x + \xi (M/N+t,
\wt a)) - \phi_{\omega(M/N,\wt a)} ( x + \xi (M/N ,  \wt a))\big)^2dx
\\  
&+  \int_0^\infty \big(\phi_{\omega(M/N-t, \wt a)} ( x + \xi (M/N-t,
\wt a)) - \phi_{\omega(M/N,\wt a)} ( x + \xi (M/N ,  \wt a))\big)^2dx 
\end{aligned}\]
and note that  
\[  \| \tau \Phi_{\ul{m}(t),
  \wt a} - \tau \Psi_\om \|_{L^2(\R)}^2\big|_{t=0}=0 \qquad \text{and} \qquad \frac{d}{dt}\| \tau \Phi_{\ul{m}(t),
  \wt a} - \tau \Psi_\om \|_{L^2(\R)}^2\bigg|_{t=0}=0.\]
To conclude the proof of \eqref{claim} it is enough to show that 
\[ \frac{d^2}{dt^2}\| \tau \Phi_{\ul{m}(t),
  \wt a} - \tau \Psi_{\om, 0} \|_{L^2(\R)}^2\bigg|_{t=0} \geq c.
\]
We compute the second derivative at $t=0$  and obtain  (we omit the
dependence of $\omega$ and $\xi$ on $m$ and $a$) 
  \begin{equation}\label{down}\begin{aligned}
&  \frac{d^2}{dt^2}\| \tau \Phi_{\ul{m}(t),
  \wt a} - \tau \Psi_\om \|_{L^2(\R)}^2\bigg|_{t=0}  \\
  = & \ 4 \int_0^\infty\left( \pd{\phi_{\omega} }{\ome}( x + \xi) \, \pd{\omega}{m}+\
  \pd{\phi_{\omega}}{x}(x+ \xi ) \,  \pd{\xi}{m}\right)^2dx \bigg|_{(m,a)=(M/N,\wt a)}. \end{aligned}\end{equation}
 Next we  prove that the integrand is not identically equal to zero,
 which in turn implies that  the second derivative in $t=0$ is
 strictly positive.
 
In Eq. \eqref{keep1} we set $z=\sqrt \omega \xi$ and take  the derivative with respect to $m$, this  gives the identity 
 \[\pd{g}{z}\, \pd{z}{m}= \frac{1}{a^{2-\mu}}, \] which tells us that $ \pd{z}{m}\neq 0$. Since 
 \[
  \pd{z}{m} = \frac{\xi}{2\sqrt\omega} \pd{\omega}{m} + \sqrt \omega \pd\xi{m}, 
 \]
 we conclude that $\pd{\omega}{m} $ and $\pd\xi{m}$ cannot be both equal to zero (recall that $\xi\neq 0$ whenever $\alpha\neq 0$). Since the functions $ \pd{\phi_{\omega} }{\ome}(x)$ and $\pd{\phi_{\omega}}{x}(x )$ are linearly independent,  the integrand in Eq. \eqref{down} does not vanish identically.

In order to prove \eqref{zwei} one
proceeds analogously. First define a curve $\Phi_{\ul{\wt m},
  a(t)}$ in 
the manifold $\M$, with 
$$
a (t) = \wt a + t, \qquad t \in (- \varepsilon,
\varepsilon).
$$
Then,
introduced the function $h(t) : = E_r (\ul{\wt m}, a
  (t))$, one immediately has
$$
h'' (0) = \f {\partial^2 E_r}{\partial a^2} ( \widetilde P
) \ = \ N \f {\partial^2 F}{\partial a^2} \left( \f M N, \wt a
\right).
$$
Hence, to prove \eqref{zwei} it is enough to show that
$h''(0)>0$. Notice that, for any $t$, the function $\Phi_{\ul{\wt m},
  a(t)}$  
is changing symmetrically on any edge,
so that, using Theorem 3.4 of \cite{gss1} again,
\begin{eqnarray*} & &
h(t) - E_r (\widetilde P) \ = \ N F \left( \f M N,
a (t)\right) - N  F \left( \f M N,
\wt a \right) \\ & = & \f N 2 \big(E_{2 \alpha/N} (\tau \Phi_{\ul{\wt m}, a(t)},\R)
- E_{2\alpha/N} (\tau \Psi_\om, \R), \big) \\
& \geq & \f N 2 c \| \tau \Phi_{\ul{\wt m}, a(t)} -  \tau \Psi_\om \|_{L^2 (\R)}^2 .
\end{eqnarray*}
To prove the inequality $h''(0)>0$ it is enough to show that  $\| \tau
\Phi_{\ul{\wt m}, a(t)} -  \tau \Psi_\om \|_{L^2 (\R)}^2 \geq
c\,t^2$ for $t$ small enough. Arguing as in the proof of
\eqref{claim}, we conclude that this is certainly true if
$\pd{\omega}{a} $ and $\pd\xi{a}$  are not both equal to zero. To see
that this is actually the case, we take  the derivative of
Eq. \eqref{keep1} with respect to $a$ and obtain an identity which is
not compatible with $\pd{z}{a}= 0$. This  in turns implies that or
$\pd{\omega}{a} \ne0 $ or $\pd\xi{a} \neq 0$.    

\end{proof}

\begin{remark} \label{alter}
As an alternative formulation of Proposition \ref{finitlocal}, we have
that $\Psi_\om$ is a {\em strict} local minimizer for the restriction of $E (\cdot,
\G)$ to the manifold $\M$.
\end{remark}

\section{Orbital stability} \label{section:orbital}
The next step consists in passing from the local minimality of
$\Psi_\om$ in $\M$ to
the local minimality of $\Psi_\om$ on $H^1_M (\G)$. This step is
immediate once one considers that: first, the reduction from $H^1_M(\G)$
to $\M$ through the transformation $\Sigma$ 
lowers the energy level; second, $\Psi_\om$ is invariant under
$\Sigma$; third, $\Sigma$ is continuous at $\Psi_\om$. Some care must
be dedicated to the fact that multiplying by a constant phase factor
does not lower the energy. Anyway, one has

\begin{proposition} \label{riduzionismo}
$\Psi_\om$ is a strict (up to multiplication by phase) local minimizer of $E(\cdot, \G)$ in $H^1_M (\G)$.
\end{proposition}

\begin{proof}
First we prove that $\Psi_\om$ is a {\em strict} local minimizer among
{\em real} functions in $H^1_M (\G)$. 
According to Remark \ref{alter},
$\Psi_{\om}$ is a strict local minimum of $E(\cdot, \G)$ restricted to $\M$.
This means
that there exists $\varepsilon > 0$ such that, if $\Phi \neq
\Psi_\omega$ is a real element
of $H^1_M (\G)$ with $\Phi (\textsc v) \neq 0$,
and $ \| \Sigma \Phi - \Psi_\om
\|_{H^1 (\G)} < \varepsilon$, then 
\be \label{pre-ineq}
E(\Psi_{\om}, \G) \leq E(\Sigma
\Phi, \G),
\ee
where equality holds if and only if $\Sigma \Phi = \Psi_\om$.

Moreover, by Rermark \ref{propsigma} we know that $\Sigma$ is
continuous at $\Psi_{\om}$, then there exists $\delta
> 0$ such that if $0 < \| \Phi - \Psi_\om
\|_{H^1 (\G)} < \delta$, then 
$$ \varepsilon \ > \ \| \Sigma \Phi - \Sigma \Psi_\om
\|_{H^1 (\G)} \ = \  \| \Sigma \Phi - \Psi_\om
\|_{H^1 (\G)}, 
$$
where we used the invariance of $\Psi_{\om}$ under the action of $\Sigma$. But
then, using  inequalities \eqref{pre-ineq} and  \eqref{energy-decrease},
$$
E (\Psi_{\om}, \G) \ \leq \ E (\Sigma \Phi, \G) \ \leq \ E (\Phi, \G),
$$
where the first inequality becomes an equality if and only if $\Sigma
\Phi = \Psi_\om$, while the second one becomes an equality if and only
if $\Phi = \Sigma \Phi$. Then, we proved that there exists a $\delta >
0$ such that if $0 < \| \Phi - \Psi_\om \|_{H^1 (\G)} < \delta$, then 
$E (\Phi, \G) < E (\Psi, \G)$,
so we have that $\Psi_\om$ strictly minimizes $E (\cdot, \G)$ locally
among the real functions in $H^1_M (\G)$. Of course, extending the
analysis to non-real functions, $\Psi_\om$ cannot be a strict
minimizer due to phase invariance of the energy functional: $E (e^{i
  \theta} \Psi_\om, \G) \ = \ E (
  \Psi_\om, \G)$. 
However, for any $\Phi$ outside the phase orbit of
  $\Psi_\om$ whose distance from the orbit is less than $\delta$, one has $E (\Phi
  , \G) \ > \ E (
  \Psi_\om, \G)$. Indeed, there exists $\theta \in [0, 2 \pi)$
    s.t. 
\be \nonumber 
   \delta >  \|
    \Phi - e^{i \theta}\Psi_\om \|_{H^1_M (\G)} =  \| e^{-i
      \theta}\Phi  - \Psi_\om \|_{H^1_M (\G)} 
\geq \| | e^{-i \theta}\Phi | -  \Psi_\om \|_{H^1_M (\G)},
\ee
so that, since $ | e^{-i \theta}\Phi |$ is real and different from $\Psi_\om$,
\be \label{real}
 E (\Psi_\om, \R) \ < \ E ( | e^{-i \theta}\Phi |, \G ).
\ee
On the other hand, a straigthforward computation gives
$$
 E ( | e^{-i \theta}\Phi |, \G )  \ \leq \ E ( e^{-i \theta}\Phi , \G ), 
$$
that, together with  \eqref{real}, concludes the proof.

\end{proof}

So we can conclude by invoking Grillakis-Shatah-Strauss theory.

\medskip

\noindent
{\em Proof of Theorem 1.} Owing to Theorem 3 in \cite{gss1}, the local
minimality property in $H^1_M (\G)$ is equivalent to orbital
stability. The proof is complete.
\qed

\end{document}